\documentclass[12pt]{amsart}

\usepackage{amsmath, amssymb, amsthm, amsthm, amscd, amsfonts}
\usepackage{mathtools}

\usepackage[dvipsnames]{xcolor}
\usepackage[alphabetic,lite,backrefs]{amsrefs}
\usepackage[colorlinks=true,hyperindex,linkcolor=magenta,pagebackref=false,citecolor=cyan]{hyperref}

\usepackage{notes}
\usepackage{cleveref}
\usepackage{enumitem}
\crefformat{equation}{(#2#1#3)} 

\usepackage{algorithm}
\usepackage{algpseudocode}
\algdef{SE}[SUBALG]{Indent}{EndIndent}{}{\algorithmicend\ }%
\algtext*{Indent}
\algtext*{EndIndent}

\let\wtilde\widetilde
\let\bar\overline

\def\cF{\mathcal F}
\def\OO{\mathcal O}
\def\FF{\mathbb F}
\def\PP{\mathbb P}
\def\ZZ{\mathbb Z}
\def\NN{\mathbb N}
\def\QQ{\mathbb Q}
\def\CC{\mathbb C}

\def\n{\mathfrak n}
\def\m{\mathfrak m}

\let\inc\hookrightarrow

\theoremstyle{theorem}
\newtheorem{thm}{Theorem}
\numberwithin{thm}{section}

\newtheorem{lem}[thm]{Lemma}
\newtheorem{cor}[thm]{Corollary}
\newtheorem{prop}[thm]{Proposition}

\theoremstyle{definition}
\newtheorem{dfn}[thm]{Definition}
\newtheorem{exa}[thm]{Example}

\newtheorem{nota}[thm]{Notation}

\newtheorem{rem}[thm]{Remark}

\title{Computing Direct Sum Decompositions}

\author{Devlin Mallory}
\address{Basque Center for Applied Mathematics, Bilbao, Basque Country, Spain}
\email{\href{mailto:dmallory@bcamath.org}{dmallory@bcamath.org}}

\author{Mahrud Sayrafi}
\address{Max Planck Institute for Mathematics in the Sciences, Leipzig, Germany}
\email{\href{mailto:mahrud@mis.mpg.de}{mahrud@mis.mpg.de}}


\begin{document}

\maketitle

\begin{abstract}
  We describe and prove correctness of two practical algorithms for finding indecomposable summands of finitely generated modules over a finitely generated $k$-algebra $R$.\linebreak The first algorithm applies in the (multi)graded case, which enables the computation of indecomposable summands of coherent sheaves on subvarieties of toric varieties (in particular, for varieties embedded in projective space); the second algorithm applies when $R$ is local and $k$ is a finite field, opening the door to computing decompositions in singularity theory.\linebreak We also present multiple examples, including some which present previously unknown phenomena regarding the behavior of summands of Frobenius pushforwards (including in the non-graded case) and syzygies over Artinian rings.
\end{abstract}

\section{Introduction}

The problems of finding isomorphism classes of indecomposable modules with a given property, or determining the indecomposable summands of a module, are ubiquitous in commutative algebra and representation theory and even in areas such as topological data analysis.\linebreak Within commutative algebra, for instance, the classification of rings $R$ for which there are only finitely many isomorphism classes of indecomposable maximal Cohen--Macaulay $R$-modules (the \emph{finite CM-type} property), or rings for which iterated Frobenius pushforwards in positive characteristic have finitely many isomorphism classes of indecomposable summands (the \emph{finite F-representation type} property) are two well-established research problems (see, for example, \cite{Yoshino90,LW12} for the former and \cite{SVdB97,Hara15,TT08}, among many others, for the latter). For both these problems, and many others, making and testing conjectures depends on computing summands of modules and verifying their indecomposability.

Currently there are no algorithms available for finding direct sum decompositions of an arbitrary module over a commutative ring. The need for and absence of such algorithms is noted, for example, in \cite[\S15.10.9]{Eisenbud95} and \cite[\S8.5]{BL23}. In contrast, variants of the ``Meat-Axe'' algorithm for determining irreducibility of finite-dimensional modules over a group algebra have wide-ranging applications in computational group theory \cite{Parker84,HR94,Holt98,IL00} and are available through symbolic algebra software such as Magma and GAP \cite{MAGMA,GAP}.
\looseness-1

Meanwhile, effective computation of indecomposable components of multiparameter persistence modules is used in topological data analysis, with applications to computational chemistry, materials science, neuroscience, and many other areas \cite{BL23}.

The purpose of this paper is to describe and prove correctness of two practical algorithms for computing indecomposable summands of finitely generated modules over a $k$-algebra $R$: \Cref{alg:graded} in the case that $M$ is a homogeneous module over a (multi)graded ring $R$ with $k$ a field of arbitrary characteristic, and \Cref{alg:local} when $R$ is local and $k$ a field of positive characteristic contained in $\bar \FF_p$. Over a large enough field, our algorithms are often capable of completely decomposing a module in a single non-recursive iteration.

In the graded case, \Cref{alg:graded} enables the computation of indecomposable summands of coherent sheaves on subvarieties of the projective space $\PP^n$ (see \Cref{sec:coherent}) as well as other complete toric varieties, while in the local case \Cref{alg:local} enables the study of germs of inhomogeneous singularities in positive characteristic (see \Cref{exa:local-singularities}).

In \Cref{sec:examples}, we present multiple examples, including  previously unknown phenomena regarding the behavior of summands of Frobenius pushforwards and syzygies over Artinian rings. In particular, we highlight the results of \cite{CDE24}, which shows a recurrence formula for indecomposable summands of high syzygies of the residue field of Golod rings, made possible through experiments and observations using our algorithm.

An implementation in Macaulay2 \cite{M2} is available via the GitHub repository \\
\centerline{
  \href{https://github.com/mahrud/DirectSummands}
       {\texttt{https://github.com/mahrud/DirectSummands}}.}
While this paper, and our implementation in Macaulay2, only concerns the case when $R$ is a commutative ring, the basic techniques may also extend to exterior algebras or Weyl algebras.

Finally, we note that there is a ``missing'' case of our algorithm: when $R$ is only local, not graded, and $k$ has characteristic 0. If a module over such a ring is decomposable, its reductions modulo $p$ will be as well. \Cref{rem:char} gives a heuristic for verifying decomposability in characteristic 0.

\subsection*{Acknowledgements}

The authors would like to thank David Eisenbud for several very useful conversations and ideas, and Ezra Miller for pointing out the relevance of this algorithm in topological data analysis. We also thank the referees for their careful reading and suggestions, which significantly improved the exposition of the paper.

The first author was supported in by the National Science Foundation under the RTG Grant No.~1840190, and by EUR2023-143443 funded by MCIN/AEI/10.13039/501100011033.
The second author was supported in part by the Doctoral Dissertation Fellowship at the University of Minnesota and the National Science Foundation under the Grant No.~2001101. Some of this work took place at the American Institute of Mathematics during the workshop ``Macaulay2: expanded functionality and improved efficiency.''

\section{Notation}

Throughout we will work over a field $k$ and write $\overline k$ for a choice of algebraic closure of $k$. \linebreak In particular, $\overline{\FF_p}$ is the algebraic closure of the finite field $\ZZ/p\ZZ$.

Let $R$ be a local or graded $k$-algebra with maximal ideal $\m$ and residue field $k = R/\m$. Given an $R$-module $M$, we write $\mu(M) \coloneqq \dim_k(M/\m M)$ for the minimal number of generators of $M$ as an $R$-module. If $R$ is graded and $M$ is homogeneous, we write $[M]_d$ for the $k$-span of its degree-$d$ elements.

We will use the Greek letters  $\phi$ or $\psi$ for elements of the endomorphism algebra $\End_R(M)$ of $M$ and capital Latin letters $A$ or $B$ for the induced $k$-linear maps on~$M/\m M$.

For an algebraic variety $X$ and a point $x\in X$, we write $k(x)=\OO_{X,x}/\m_x$ for the residue field of $X$ at $x$ and $k(X)$ for the function field of $X$.
When $X$ is a projective variety and $\OO_X(1)$ is a fixed very ample line bundle on $X$, the twisted global sections functor $\Gamma_*$ is the functor given by $\cF \mapsto \bigoplus_{n\in\ZZ} H^0(X, \cF(n))$, sending coherent sheaves to modules over the graded ring $S = \Gamma_*(\OO_X)$. Conversely, if $M$ is a finitely generated, graded $S$-module, we write $\wtilde M$ for the coherent sheaf associated to $M$.

\section{The main algorithms}

We begin by describing the main algorithms for finding the indecomposable summands of a finitely generated module over a graded ring or a local ring with residue field contained in $\overline{\FF_p}$.

\subsection{The graded case}\label{sec:graded-alg}

Let $R$ be a $\NN$-graded ring with $R_0 = k$ a field and $M = \bigoplus [M]_d$ a finitely generated, graded $R$-module and consider a degree-zero endomorphism $\phi \in [\End_R(M)]_0$. Fix a set of minimal homogeneous generators $m_1,\dots,m_{\mu(M)}$ of $M$ so that $\phi$ may be presented as a $\mu(M)\times\mu(M)$ matrix with entries in $R$. We will refer to $\chi_\phi(\lambda) = \det(\phi-\lambda\id_M)$ as the ``characteristic polynomial of $\phi$'', which is a univariate polynomial with coefficients in $R$ and thus a priori has solutions only in the algebraic closure of the total ring of fractions of $R$.

\begin{prop}\label{prop:eigenvalues}
  Let $A\colon M/\m M\to M/\m M$ be the $k$-linear map induced by $\phi\in[\End_R(M)]_0$. The solutions of the characteristic polynomial $\chi_\phi(\lambda)$ are the same as the eigenvalues of $A$, thus they lie in the algebraic closure $\bar k$.
\end{prop}
\begin{proof}
  Since $M$ is $\NN$-graded, we may order its generators such that $\deg m_i$ is nonincreasing; let $d_1,\dots,d_n$ be the decreasing list of unique degrees of generators of $M$.
The matrix for $\phi$ is then ``block upper triangular'' with $n$ diagonal blocks with entries in $k$, i.e., of the form
  \begin{equation}\label{eq:blocks}
    \phi = \begin{pmatrix}
    B_1    & *      & \dots  & *      \\
    0      & B_2    & \dots  & *      \\
    \vdots & \vdots & \ddots & \vdots \\
    0      & 0      & \dots  & B_n
    \end{pmatrix},
  \end{equation}
  where the blocks $B_i$ are square matrices \emph{with entries in $k$}, and the remaining entries are in $\m$. To see this, note that since $\phi$ has degree zero, $\phi(m_i)$ is the sum of a $k$-linear combination of the $m_j$ of the same degree as $m_i$ and an $\m$-linear combination of the $m_j$ of lower degree. 

  Thus the characteristic polynomial $\det(\phi - \lambda \id_M) = \prod \det(B_i - \lambda\id) = \det(A - \lambda \id_{M/\m M})$ is a univariate polynomial in $\lambda$ with coefficients in $k$, which takes solutions in $\bar k$.
\end{proof}

\begin{dfn}\label{def:eigenvalues}
  The collection of $k$-eigenvalues of $B_i$ are the \emph{eigenvalues of $\phi$}. We will write $\mu_j \coloneqq \mu(\lambda_j)$ for the sum of the geometric multiplicities of $\lambda_j$ in each $B_i$. Note that the algebraic multiplicity of $\lambda_j$ in $\chi_\phi(\lambda)$ is the sum of the algebraic multiplicities of $\lambda_j$ in each $\chi_{B_i}(\lambda)$ and may be larger than $\mu_j$.
\end{dfn}

\begin{rem}\label{rem:grading}
  The hypotheses that $R$ is $\NN$-graded may be weakened. In the proof of \Cref{prop:eigenvalues}, we only need a partial order on the degrees of the generators of $M$ in order to present $\phi$ in a block upper triangular form \eqref{eq:blocks}. If $R$ is $\ZZ^r$-graded, then the \emph{effective cone} $\Eff(R)\subset\QQ^n$ generated by the elements $\deg m$ for each monomial $m \in R$ induces a partial order on the degrees of generators of $M$ if it is strongly convex (contains no positive dimensional subspace of $\QQ^n$) and $R_0 = k$. We will call such rings \emph{positively graded}. 
  This is satisfied, for instance, in the setting of multiparameter persistence modules or when $R$ is the Cox ring of a complete toric variety or Mori dream space graded by its class group.
\end{rem}

We will use nonzero eigenvalues of $\phi$ to get nontrivial splittings of $M$. As the following lemma shows, this approach allows for decomposing multiple summands at once.

\begin{lem}\label{lem:decomposition}
  Suppose $\psi_1,\dots,\psi_r$ are endomorphisms of $M$ such that $M\to\im\psi_i$ is a split surjection for all $i$ and set $\psi = \psi_r\circ\dots\circ\psi_1$.
  If $\ker\psi_i\cap\ker \psi_j = 0$ for all $i\neq j$ then $M$ has a direct sum decomposition \( M = \ker\psi_r \oplus \cdots \oplus \ker\psi_1 \oplus \im\psi. \)
\end{lem}
\begin{proof}
  Since each $\psi_i$ is a split surjection, there are direct sum decompositions $M=\im \psi_i \oplus \ker \psi_i$, and thus there exist idempotents $e_i\in \End M$ with $\im e_i = \im \psi_i$ and $\ker e_i = \ker \psi_i$.

  The $e_i$ must commute: let $f_i = 1- e_i$ be the complementary idempotent (projection onto $\ker \psi_i$). Since $\ker \psi_i \cap \ker \psi_j=0$, $f_if_j=0$. In particular,
  \[ e_ie_j =(1-f_i)(1-f_j) = 1-f_i-f_j= 1-f_j-f_i = e_j e_i. \]

  Since the $e_i$ are commuting idempotents, $e_r\cdots e_1$ is also idempotent. It's clear that $\im(e_r\cdots e_1) = \im (\psi_r\circ \dots\circ \psi_1) = \im \psi$. Moreover, we claim that $\ker(e_r\cdots e_1) = \ker \psi_r \oplus \cdots \oplus \ker \psi_1$. To see this, note that since $f_if_j=0$ we get
  \[ 1-e_r\cdots e_1 = 1-(1-f_r)\cdots (1-f_1) = \sum f_i. \]
  In other words, the kernel of $e_r\cdots e_1$ is exactly $\bigoplus \ker e_i =\bigoplus \ker \psi_i$.

  Thus, the idempotent $e_r\cdots e_1$ corresponds to a direct sum decomposition
  \[ M = \ker(e_r\cdots e_1) \oplus \im(e_r\cdots e_1) = \ker\psi_r \oplus \cdots \oplus \ker\psi_1 \oplus \im\psi, \]
  as desired.
\end{proof}

Note that while $\im\psi$ may be zero, as long as each $\psi_i$ in \Cref{lem:decomposition} is not zero or an isomorphism, the kernel summands will be nontrivial (although they may not necessarily be indecomposable).

The following proposition is a module-theoretic analogue of the primary decomposition theorem in linear algebra:

\begin{prop}\label{prop:split-surj}
  Let $\lambda_1,\dots,\lambda_r$ be distinct eigenvalues of $\phi$ and set $\psi_j = (\phi - \lambda_j\id_M)^{\mu_j}$ and $\psi = \psi_1\circ\cdots\circ\psi_r$. Then $M$ has a direct sum decomposition
  \( M = \ker\psi_1 \oplus \cdots \oplus \ker\psi_r \oplus \im\psi. \)
\end{prop}

\begin{proof}
  If $\psi_i$ is injective on $\im\psi_i$, then $\im\psi_i$ is a summand: in this case, the composition $\im\psi_i\hookrightarrow M\to \im\psi_i$ is injective, but such a degree-zero injective endomorphism of a finitely generated $R$-module is an isomorphism (it must be injective, hence surjective, on each finite-dimensional $[M]_d$). Thus, the inclusion $\im \psi \to M$ is split by the surjection $M\to \im \psi$, and so the short exact sequence $0\to\ker\psi_i\to M\to\im\psi_i\to 0$ is split and $M = \ker\psi_i\oplus\im\psi_i$.

  To see that $\psi_i$ is injective on $\im\psi_i$, we can pass to $\overline k$ and thus assume all eigenvalues of $\phi$ exist in $k$. For simplicity, replace $\phi-\lambda_i\id_M$ by $\phi$ so $\psi_i = \phi^{\mu_i} $ divides $ \chi_\phi(\phi)$.
By the Cayley--Hamilton theorem \cite[Theorem~4.3]{Eisenbud95},
  \begin{align*}
    \chi_\phi(\phi) = \left(\phi^{\mu(M) - \mu_i} - \dots + c_1\phi - c_0\right) \circ \psi_i
    &= \left(\phi^{\mu(M) - \mu_i - 1} - \dots + c_1\right) \circ \phi \circ \psi_i - c_0 \psi_i = 0,
  \end{align*}
  for some $c_0,c_1$ in $k$. Moreover, $c_0\neq 0$, because otherwise $\chi_\psi(\psi)$ is divisible by $\phi^{\mu_i+1}$, contradicting the fact that $\mu_i$ is the algebraic multiplicity of $\lambda_i$ in $\chi_\phi(\phi)$.
Thus,
  \( \psi_i = c_0^{-1} \left(\phi^{\mu(M) - \mu_i - 1} - \dots + c_1\right) \circ \phi \circ \psi_i. \)
  Thus $\phi$ is injective on $\im\psi_i$, and therefore the composition $\psi_i=\phi^{\mu_i}$ is still injective on $\im\psi_i$.
We thus have that $\psi_i$ is a split surjection, as desired.

  To see that $\ker\psi_i\cap\ker\psi_j = 0$ if $i\neq j$, we mimic the standard proof for vector spaces: suppose $(\phi - \lambda_i)^m(v) = 0$ for some $v\in M$ and $m$ is minimal, so that $w = (\phi-\lambda_i)^{m-1}(v) \neq 0$ satisfies $(\phi-\lambda_i)(w) = 0$.
If $(\phi - \lambda_j)^{\mu_j}(v) = 0$ as well, then since $(\phi - \lambda_j)(\phi - \lambda_i) = (\phi - \lambda_i)(\phi - \lambda_j)$ (as $\phi$ commutes with itself and the identity) we get
  \[
  (\phi - \lambda_j)^{\mu_j}(w)
  = (\phi - \lambda_j)^{\mu_j}(\phi - \lambda_i)^{m-1}(v)
  = (\phi - \lambda_i)^{m-1}(\phi - \lambda_j)^{\mu_j}(v) = 0.
  \]
Thus,
  \(
  0 = (\phi - \lambda_j)^{\mu_j}(w) = (\phi - \lambda_i + \lambda_i - \lambda_j)^{\mu_j}(w) = (\lambda_i - \lambda_j)^{\mu_j}(w)
  \)
(using that $(\phi - \lambda_i)(w)=0$).
Since $\lambda_i-\lambda_j \in k$ is nonzero,
$(\lambda_i - \lambda_j)^{\mu_j}(w)= 0 $ forces $w=0$, a contradiction.

  Finally, \Cref{lem:decomposition} implies the desired decomposition.
\vadjust{\goodbreak}
\end{proof}

We are ready to state the algorithm for decomposition of a graded module. Recall from \Cref{rem:grading} that we only need the ring $R$ to be positively graded over an arbitrary field.

\begin{algorithm}[H]
  \caption{(Indecomposable summands of a graded module over a commutative ring)}\label{alg:graded}
  \begin{algorithmic}[1]
    \smallskip
    \Require graded module $M$ over a positively graded commutative $k$-algebra $R$.
    \Ensure  list $\textproc{Summands}(M)$ of indecomposable summands of $M$.
    \State Initialize a list $L$.
    \State Take a general element $\phi$ in $[\End_R(M)]_0$,
           the degree-zero part of $\End_R(M)$. \label{item:End0}
    \State Let $k'$ be the splitting field of the minimal polynomial of $\phi$.
    \ForAll { $k'$-eigenvalues $\lambda_i$ of $\phi$ }
      \State Let $\mu_i$ be the (geometric) multiplicity of $\lambda_i$.
      \State Let $\psi_i = (\phi - \lambda_i \id_M)^{\mu_i}$ (or, for simplicity, raised to $\mu(M)$).
      \State By \Cref{prop:split-surj}, $\psi_i$ is a split surjection, hence producing a splitting of $M$.
      \State Append \Call{Summands}{$\ker\psi_i$} to $L$.
    \EndFor
    \State Let $\psi$ be the composition of all $\psi_i$ above.
    \If {$\psi$ is nonzero}
      \State Append \Call{Summands}{$\im\psi$} to $L$.
    \EndIf
    \State \Return $L$.
  \end{algorithmic}
\end{algorithm}

We will see via \Cref{lem:general} that this results in an indecomposable decomposition of $M$.

\begin{rem}\label{rem:graded}
  Note that \Cref{alg:graded} does not depend on the characteristic of the field. Further, since over a large enough field (i.e., when $|k|$ is larger than the number of summands) the eigenvalues of a general endomorphism $\phi$ are as distinct as  possible (by a variant of \Cref{lem:general}), the summands $\ker\psi_i$ and $\im\psi$ will be indecomposable; thus \Cref{alg:graded} will produce an indecomposable decomposition of $M$ in one non-recursive iteration. When $|k|$ is less than or approximately the same as the number of summands, the algorithm will be recursive.
\looseness-1
\end{rem}

\subsection{The local case}\label{sec:local-alg}

Throughout this section, $R$ will be a local ring over a field of positive characteristic $k \subset \overline{\FF_p}$ with maximal ideal $\m$ and $M$ a finitely generated $R$-module. Recall that if $\psi\in\End_R(M)$ is an idempotent, then $M$ decomposes as $\im \psi \oplus \coker \psi$. If $\psi$ is not zero or an isomorphism, then both factors are nonzero and $M$ is decomposable.

We begin with the observation that $\psi$ also acts on the $k$-vector space $M/\m M$. The following lemma allows us to check only for idempotents modulo the maximal ideal.

\begin{lem}\label{lem:idemp}
  Let $A\colon M/\m M\to M/\m M$ be the induced action of $\psi\in\End_R(M)$ on $M/\m M$. If $A$ is an idempotent, then $M$ admits a direct sum decomposition $\ker\psi \oplus \im\psi$.
\end{lem}
\begin{proof}
  Let $N \coloneqq \im\psi$. We want to show that $0 \to N \inc M$ splits. Consider the composition
  \begin{equation}\label{eq:N-M-N}
    N \inc M \xra{\psi} N.
  \end{equation}
  We claim that this composition is surjective. Since a surjective endomorphism of finitely generated modules is invertible \cite[Corollary~4.4]{Eisenbud95}, we conclude that this composition is an isomorphism $\alpha$ on $N$.
  Therefore the inclusion $ 0 \to N \inc M $ is split by the composition
  \[ M \xra{\psi} N \xra{\alpha\inv} N, \]
  and thus $M$ decomposes as claimed.

  To check the surjectivity of \eqref{eq:N-M-N}, we may complete at the maximal ideal: by Nakayama's lemma, it suffices to check surjectivity of the map modulo $\mathfrak m$, and completing and reducing modulo $\mathfrak m$ is the same as just reducing modulo $\mathfrak m$. We may thus assume $R$, $M$, and $N$ are complete. Observe that by assumption we can write $\psi^2 = \psi + \psi'$ with $\psi'(M)\subset \m M$.
  Note that if $x \in \m^\ell M$, then $\psi'(x) = \psi^2(x) - \psi(x)$ lies in $\m^{\ell+1}M$.

  Let $n_0\in N$. By assumption, $n_0 = \psi(m_1)$ for some $m_1\in M$. Applying $\psi$ again, we get
  \[ \psi(n_0) = \psi^2(m_1) = \psi(m_1) + \psi'(m_1) = n_0 + n_1, \]
  where $n_1 = \psi^2(m_1) - \psi(m_1) = \psi(n_0) - n_0 \in \m M$. In fact, since $\psi(n_0)$ and $n_0$ are both in $N$, we have $n_1\in N$ as well, so $n_1\in \m M\cap N$. Thus, we can write $n_1 = \psi(m_2)$ for $m_2\in M$.
  Now, apply $\psi$ to both sides: by the assumption that $\psi$ is idempotent modulo $\m$, we have
  \[ \psi(n_1)=\psi^2(m_2) = \psi(m_2) + \psi'(m_2) = n_1 + n_2, \]
  where, similarly, $n_2 = \psi(n_1) - n_1$, so $n_2 \in \m^2M\cap N$ as well.
  Combining, we can write
  \[ n_0 = \psi(n_0) - n_1 = \psi(n_0) - \psi(n_1) + n_2 = \psi(n_0 - n_1) + n_2, \]
  with $n_1\in \m M\cap N$ and $n_2\in \m^2M\cap N$. Continuing in this fashion, for any $\ell$ we can write
  \[ n_0=\psi(n_0-n_1+\dots \pm n_{\ell}) \mp n_{\ell+1}, \]
  with $n_i \in \m^i M\cap N$.

  By the Artin--Rees lemma \cite[Lemma~5.1]{Eisenbud95}, there is $\ell>0$ such that for $i\gg0$ we get
  \[ \m^i M\cap N = \m^{i-\ell} ( \m^i M\cap N)\subset \m^{i-\ell} N. \]
  That is, the terms of $n_0-n_1+\cdots$ go to 0 in the $\m$-adic topology on $N$. Thus we can write
  \[ n_0 = \psi(n_0-n_1+n_2-\cdots), \]
  with $n_0-n_1+n_2-\cdots\in N$. We conclude that $\psi$ is surjective as a map $N\to N$.
\end{proof}

Thus, if we produce an element $\psi \in \End_R(M)$ that is an idempotent modulo $\m$, we obtain a splitting of $M$.
The following lemma allows us to produce idempotents modulo $\m$.
\vadjust{\goodbreak}

\begin{lem}\label{lem:jordan}
  Let $k\subset \bar \FF_{p}$ be a subfield and $A$ an endomorphism of a $k$-vector space.
Choose:
\begin{itemize}[topsep=6pt,itemsep=0pt]
\item $e_0$ such that $p^{e_0}$ is larger than the geometric multiplicity of any eigenvalue of $A$ in $\bar k$.
\item $e$ such that all eigenvalues of $A$ in $\bar k$ lie in $\FF_{p^e}$.
\end{itemize}
If $\lambda$ is an eigenvalue of $A$ contained in $k$, then $(A-\lambda)^{p^{e_0} (p^{e}-1)}$ is idempotent.
  If $\lambda$ is not the only eigenvalue of $A$ over $\bar k$, this is a nontrivial idempotent.
\end{lem}
\begin{proof}
To check that this power is idempotent, we may extend to the algebraic closure $\bar k$, because if $A^2=A$ as an endomorphism of $\bar k$-vector spaces the same is true over $k$.
  Since all eigenvalues of $A$ are contained in $\bar k$, we can without loss of generality put $A$ in Jordan canonical form, with each Jordan block being an $r_i\times r_i$ matrix of the form
  \[ \begin{pmatrix}
        \lambda_i & 1 & 0 & \dots & 0 \\
        0 & \lambda_i & 1  & \dots & 0 \\
        \vdots & \vdots & \vdots & \ddots & \vdots \\
        0 & 0 & 0 & \dots & \lambda_i
    \end{pmatrix}, \]
  where each $\lambda_i$ is an eigenvalue of $A$ and $r_i$ is bounded above by the geometric multiplicity of $\lambda_i$ in $A$.
  In this basis, $A-\lambda$ will be block-diagonal with blocks
  \[ \begin{pmatrix}
    \lambda_i-\lambda & 1 & 0 & \dots & 0 \\
    0 & \lambda_i-\lambda & 1 & \dots & 0 \\
    \vdots & \vdots & \vdots & \ddots & \vdots \\
    0 & 0 & 0 & \dots & \lambda_i-\lambda
  \end{pmatrix}. \]
  Set $\nu_i=\lambda_i-\lambda$. Then for any $n\geq 1$,
  the $n$-th power $(A-\lambda)^n$ is block-diagonal with blocks
  \[ \begin{pmatrix}
    \nu_i^n & \binom{n}1\nu_i^{n-1} & \binom{n}2\nu_i^{n-2} & \dots & \binom{n}{r_i}\nu_i^{n-r_i} \\
    0 & \nu_i^n & \binom{n}1\nu_i^{n-1} & \dots & \binom{n}{r_i-1}\nu_i^{n-r_i+1} \\
    \vdots & \vdots & \vdots & \ddots & \vdots \\
    0 & 0 & 0 & \dots & \nu_i^n
  \end{pmatrix}. \]
  If we choose $n > r_i$ and equal to a power of $p$, for example, $p^{e_0}$, then all non-diagonal terms will vanish by Lukas' theorem \cite[Lemma~15.22]{Eisenbud95}, so all blocks will have the form
  \[ \begin{pmatrix}
    \nu_i^n & 0 & 0 & \dots & 0 \\
    0 & \nu_i^n & 0 & \dots & 0 \\
    \vdots & \vdots & \vdots & \ddots & \vdots \\
    0 & 0 & 0 & \dots & \nu_i^n
  \end{pmatrix}. \]
  Finally, since $\nu_i =\lambda-\lambda_i \in \FF_{p^e}$ for each $i$, if we choose $n$ to be divisible also by $p^e-1$, we have for each nonzero $\nu_i$ that $$\nu_i^{n} = (\nu_i^{p^{e}-1})^{n/(p^{e}-1)} =1.$$
(We use here that $x^{p^e-1}=1$ for all $x\in \FF_{p^e}^\times$, since $\F_{p^e}^\times$ is a group of order $p^e-1$.)
  In particular, taking $n=p^{e_0}(p^e-1)$, the matrix $(A-\lambda)^{p^{e_0}(p^e-1)}$ is a diagonal matrix with diagonal entries 1 or 0, hence idempotent. Moreover, if some $\lambda_i\neq \lambda$ then $(A-\lambda)^n$ is not the zero matrix.
\looseness-1
\end{proof}

\begin{prop}\label{prop:split-idems}
  Let $\lambda_1,\dots,\lambda_r$ be the eigenvalues of $\phi\otimes R/\m$ defined over $k$. Pick $e$ such that all eigenvalues of $\phi \otimes R/\m$ in $\bar k$ are contained in $\FF_{p^e}$ and $e_0 =  \lceil \log_p(\mu(M) + 1)\rceil$ and set $\psi_j = (\phi - \lambda_j \id_M)^{p^{e_0} (p^{e} - 1)}$
and $\psi = \psi_1\circ\cdots\circ\psi_r$. Then $M$ has a direct sum decomposition
  \( M = \ker\psi_1 \oplus \cdots \oplus \ker\psi_r \oplus \im\psi. \)
\end{prop}

We note that $e$ can be calculated effectively by computing the splitting field of $\chi_{\phi\otimes R/\m}(\lambda)$.

\begin{proof}
  This is analogous to the proof of \Cref{prop:split-surj}. By \Cref{lem:jordan} and \Cref{lem:idemp},
since $p^{e_0}> \mu(M)$ is an upper bound on the geometric multiplicity of any eigenvalue in $\bar k$,
 each $\psi_i$ is idempotent and in particular a split surjection. Further, $\ker\psi_i\cap\ker\psi_j = 0$ by the same argument as in the proof of \Cref{prop:split-surj}. Therefore, using \Cref{lem:decomposition}, we have the desired direct sum decomposition.
\end{proof}

\pagebreak

This leads to a probabilistic algorithm to find the indecomposable summands of a finitely generated $R$-module $M$ in our setting, as follows:

\begin{algorithm}[H]
  \caption{(Indecomposable summands of a module over a commutative local ring)}\label{alg:local}
  \begin{algorithmic}[1]
    \smallskip
    \Require module $M$ over a commutative local ring $R$ such that $R/\m \subset\bar\FF_p$
    \Ensure  list $\textproc{Summands}(M)$ of indecomposable summands of $M$.
    \State Initialize a list $L$.
    \State Take a general element $\phi$ in $\End_R(M) \setminus \m \End_R(M)$.
    \State Let $A$ be the induced endomorphism of the $k$-vector space $M/\m M$.
    \State Let $k'$ be the splitting field of the minimal polynomial of $A$.
    \State Let $e_0 = \lceil\log_p(\mu(M) + 1)\rceil$, where $\mu(M) \coloneqq \dim_{R/\m}(M/\m M)$.
    \State Let $e$ be the least integer such that $\chi_{\phi\otimes R/\m}(\lambda)$ factors into linear terms in $\FF_{p^e}$.
    \ForAll { $k'$-eigenvalues $\lambda_i$ of $A$ }
      \State Let $\psi_i = (\phi - \lambda_i \id_M)^{p^{e_0} (p^{e} - 1)}$.
      \State By \Cref{prop:split-idems}, $\psi_i$ is idempotent, hence producing a splitting of $M$.
      \State Append \Call{Summands}{$\ker\psi_i$} to $L$.
    \EndFor
    \State Let $\psi$ be the composition of all $\psi_i$ above.
    \If {$\psi$ is nonzero}
      \State Append \Call{Summands}{$\im\psi$} to $L$.
    \EndIf
    \State \Return $L$.
  \end{algorithmic}
\end{algorithm}


\begin{rem}\label{rem:local}
  Similarly as in \Cref{rem:graded}, over a large enough field \Cref{alg:local} will produce indecomposable summands in one non-recursive iteration, and otherwise it will recursively split the summands until it finds an indecomposable decomposition of $M$.
\end{rem}

\begin{rem}\label{rem:extension}
  Note that \Cref{alg:graded,alg:local} are quite sensitive to the ground field $k$, as they need $k$-eigenvalues of a general endomorphism, and may produce splittings over a finite extension of $k$, even if the module splits already over $k$. 
We will calculate the probability of picking an endomorphism with distinct eigenvalues in \Cref{lem:general}, and show that if $|k|$ is sufficiently large this probability is close to 1.
  For an example where extending the base field is necessary, see \Cref{ex:elliptic}.
\end{rem}

\begin{rem}\label{rem:char}
  At the moment, there is not a splitting algorithm for the case when $R$ is a local ring over a field $k\subset\CC$ of characteristic 0. However, we note that \Cref{alg:local} can be used to test indecomposability of $M$ via reduction modulo~$p$, as follows.
  One can choose a finitely generated $\ZZ$-algebra $T$ and an $T$-algebra $R_T$ such that $R_T\otimes_T k = R$, and likewise an $R_T$-module $M_T$ such that $M_T\otimes _T k = M$ and $M_T$ is flat over $T$. If $\n$ is a maximal ideal of $T$, one can check that $T/\n \cong \FF_{p^e}$ for some prime $p$ and~$e$.
  The key point is that if $M$ is decomposable, then we can enlarge $T$ (if necessary) such that $M_T$ is decomposable. Thus, the various reductions  $M\otimes T/\n$ will also be decomposable for all $\n$.
  Picking a maximal ideal $\n$ of $T$, if our algorithm does not detect an indecomposable summand of $M\otimes T/\n$, then the original module $M$ must have been indecomposable.
  It would be interesting to investigate whether the Hasse principle holds in this case; namely, whether the decompositions of various $M_T\otimes_T T/\n$ can be patched into a decomposition of $M$.
\looseness-1

  We also point out that it is possible to ``guess'' an idempotent for $M$, even when there is no algorithm to produce one.  Since idempotents of $M$ are never in $\m\End_R (M)$, Macaulay2 often chooses them as some of the minimal generators of $\End_R(M)$. We can often produce nontrivial direct sum decompositions in the local case in characteristic 0,
by checking if the minimal generators of $\End_R(M)$ are idempotent.
\end{rem}

\subsection{Termination of the algorithms}

Throughout this section, let $M$ be a finitely generated $R$-module over a graded ring $R$ with $R_0 = k$ (in this case $M$ should also be graded) or over a local ring $(R,\m)$ with $k = R/\m\subset \bar \FF_p$. As a consequence of \Cref{alg:graded,alg:local}, we have:

\begin{lem}\label{lem:distinct}
  If there is $\phi\in\End_R(M)$ with two distinct eigenvalues, then $M$ is decomposable.
\end{lem}

Conversely, the following lemma implies that if $M$ is decomposable, then \Cref{alg:graded,alg:local} will find a decomposition of $M$ as a direct sum of indecomposable modules.

\begin{lem}\label{lem:general}
  Suppose $M$ is a finitely generated $R$-module which is decomposable.
  \begin{itemize}[topsep=6pt,itemsep=0pt]
  \item If $R$ is graded, let $\phi$ be a general $ k$-linear combination of a $k$-basis for $[\End_R(M)]_0$.
  \item If $R$ is local, let $\phi$ be a general $ k$-linear combination of minimal generators of $\End_R(M)$.
  \end{itemize}
  Let $L$ be the distinct $\bar k$-eigenvalues of $\phi$ in the graded case, or $\phi\otimes R/\m$ in the local case.
  If $k$ is infinite, then $|L| \geq 2$. If $k$ is finite, then the probability that $|L| \geq 2$ is at least
  $1 - \mu(M) \cdot \frac{|k|^{r-1}-1}{|k|^{r}-1}$, which approaches $1$ as $|k| \to \infty$.
\end{lem}


We will refer to such endomorphisms of $M$ as \emph{general}. In particular, if $M$ is decomposable, then \Cref{alg:graded,alg:local} produce a nontrivial direct sum decomposition of $M$, because a general element of $\End_R(M)$ will have distinct eigenvalues and thus the algorithms will produce a nontrivial decomposition. We emphasize that even if $M$ splits over $k$, a general element of $\End_R(M)$ may have eigenvalues over an extension of $k$ (see \Cref{rem:extension}).

\begin{nota}\label{notation:gensEnd0}
  Let $\phi_1,\dots,\phi_r$ form a $k$-basis for $[\End_R(M)]_0$ in the graded case, or minimal generators of $\End_R(M)$ in the local case, and let $A_1,\dots,A_r$ be their images modulo $\m$, meaning $k$-linear transformations $M/\m M\to M/\m M$ viewed as matrices with entries in $k$.
\end{nota}

\begin{proof}[Proof of \Cref{lem:general}]
  Let $U\subset\AA_k^r$ be the subset of $r$-tuples $(\alpha_1,\dots,\alpha_r)$ such that $\sum_i\alpha_i A_i$ has at least two distinct eigenvalues, i.e., such that $\sum_i\alpha_i\phi_i$ reduces to an endomorphism of $M/\m M$ with at least two distinct eigenvalues. 
  \vadjust{\goodbreak}

  It suffices to show that $U$ is a nonempty open subset of $\AA_k^r$. First, we show $U$ is nonempty:
  say $M = M_1\oplus M_2$ is a nontrivial decomposition and note that we may take $\phi_1,\phi_2$ to be projectors onto each summand. Then for any $\alpha_1,\alpha_2\in k$ the combination $\alpha_1 A_1 + \alpha_2 A_2$ has eigenvalues $\alpha_1,\alpha_2$; thus in particular there is \emph{an} element of $\End_R(M)$ reducing to an endomorphism of $M/\m M$ with distinct eigenvalues, so $U$ is nonempty.

  Now, we show that $U$ is open. This is a purely linear algebraic statement: Given any $r$ $m\times m$ matrices $A_1,\dots,A_r$, we claim that $A \coloneqq \sum_i\alpha_i A_i$ has at least two distinct eigenvalues for $\alpha_1,\dots,\alpha_r$ outside a Zariski-closed subset of $\AA_k^r$.
  The eigenvalues of $A$ are the roots of
  $$ \chi_{A}(\lambda) = \det(A-\lambda \id_{M/\m M}), $$
  which is a polynomial in $\lambda$ with coefficients in $\alpha_1,\dots,\alpha_r$. Then $A$ fails to have at least two distinct eigenvalues exactly when this polynomial factors as a pure power of a linear term.

  This condition is polynomial in the coefficients of powers of $\lambda$ in
  \begin{equation}\label{eq:roots}
    f \coloneqq \chi_{A}(\lambda) = \lambda^m + c_{m-1} \lambda^{m-1} + \dots + c_1 \lambda + c_0,
  \end{equation}
  with $c_j$ a polynomial of degree $m-j$ in the $\alpha_i$. In particular, $f$ has an $m$-fold root when
  $$ f = \frac{\partial f}{\partial\lambda} = \dots = \frac{\partial^{m-1}f}{\partial\lambda^{m-1}} = 0 $$
  vanish simultaneously. The resultant of these $m$ polynomials in the $m$ variables $c_i$ gives polynomial conditions in the $c_i$ for this to occur.
  In our setting, the $c_i$ are themselves polynomials in the $\alpha_i$, and thus we have obtained polynomial equations defining the locus where $A$ fails to have distinct roots, and thus the complement $U$ is open.

  Over an infinite field, then, a general choice of $k$-linear combination of the $\phi_i$ will lie in the nonempty open set $U$, i.e., will have at least two distinct eigenvalues.

  All that remains is to estimate the proportion of points of $\AA_k^r$ that lie in $U$ when $k$ is finite.
  That is, we want to compare the number of $\F_q$-points $(\alpha_i) \in U$ (i.e., the endomorphism with at least distinct eigenvalues) to the total number of points in $\AA_k^r$.
  We will give a very rough lower bound on the number of $\F_q$-points of $U$, by finding a single nonzero polynomial vanishing on $\AA_k^r\setminus U$.

  First, we find conditions on the coefficients $c_i$ of the polynomial $\chi_A(\lambda)$ for $f$ to be a power of a linear form. Say there is $y \in k$ such that
  $$ \chi_{A}(\lambda) = \lambda^m + c_{m-1} \lambda^{m-1} + \dots + c_1 \lambda + c_0 = (\lambda-y)^m. $$
  Equating coefficients, we have
  $$ c_{m-i} = (-1)^i \binom{m}{i} y^i. $$
  If $p = \operatorname{char} k\nmid m$, so $m^i\neq 0$ for any $i$, the coefficients $c_{m-i}$ then satisfy the $m$ (nontrivial) equations of the form
  \begin{equation}\label{eq:goodchar}
    \binom{m}{m-i} \cdot c_{m-1}^i = \pm m^i\cdot c_{m-i}.
  \end{equation}
  (When $\Char k = 0$, this recovers the ideal defining the rational normal curve of degree $m$; see for example \cite{Chi03,Chi04}.)
  If instead $m = p^d m_0$, then the only possible nonzero coefficients are $c_{m-ip^d}$ for $i=1, \dots, m_0$ (by Lukas' theorem \cite[Lemma~15.22]{Eisenbud95}), and these satisfy the $m_0$ (nontrivial) equations
  \begin{equation}\label{eq:badchar}
    \binom{m}{m-ip^d} \cdot c_{m-p^d}^i =\pm \binom{m}{m-p^d}^i c_{m-i p^d}
  \end{equation}
  and the $m-m_0$ equations $c_i = 0$ for $p^d \nmid i$.

  Note that the equations in the $c_i$ vanish on the set of matrices with exactly one eigenvalue; it is also straightforward
  to verify that if these equations are satisfied for some choice of $c_i$, then $\chi_A(\lambda)$ equals $(\lambda -y)^m$ for some $y \in k$.  Thus these equations, viewed as equations in the variables $c_i$, define the locus of characteristic polynomials that factor as a power of a linear form (at least up to radical, one can verify that this ideal is in fact radical).

  Since the $c_i$ are themselves polynomials in the $\alpha_i$, we have a set of equations in the $\alpha_i$ that define the locus of endomorphisms with only one eigenvalue.
  Note that since $c_{m-i}$ is a degree-$i$ polynomial in the $\alpha_i$, these equations are homogeneous of degree $\leq m$ in the $\alpha_i$.
  Moreover, since there is an endomorphism with distinct eigenvalues (because we assumed $M$ is decomposable), the resulting homogeneous equations in the $\alpha_i$ are not all zero.

  Denote by $Z$ the locus of points such that $f$ has an $m$-fold root, viewed as a subset of $\P_k^{r-1}$.
  (As the property of $\sum \alpha_i A_i$ having distinct roots is invariant under simultaneously scaling all $\alpha_i$ by a nonzero scalar, we can view the $\alpha_i$ as homogeneous coordinates on $\P_k^{r-1}$; the proportion of elements with distinct eigenvalues is the same whether we view our choice in $\AA_k^r$ or $\P_k^{r-1}$.)
  By the discussion of the preceding paragraph, there is a homogeneous equation in the $\alpha_i$ of degree $\leq m$ which vanishes on $Z$.

  $Z$ is contained in an $(r-2)$-dimensional hypersurface of degree $\leq m$.
  We can thus apply a quantitative form of the Lang--Weil estimate \cite[Proposition~12.1]{GL02} to conclude that
  $|Z(k)| \leq m \cdot  |\P_k^{r-2}(k)|=m \frac{|k|^{r-1}-1}{|k|-1}$.
  Thus, the proportion of endomorphisms with exactly one eigenvalue is bounded above by
  $$ m \cdot \frac{ |\P_k^{r-2}(k)|}{|\P_k^{r-1}(k)|} = m \cdot \frac{|k|^{r-1}-1}{|k|^{r}-1} $$
  Thus, the proportion of endomorphisms with at least two eigenvalues is at least
  \[ 1 - m \cdot \frac{|k|^{r-1}-1}{|k|^{r}-1} . \hfill \qedhere \]
\end{proof}

\begin{rem}
  As $|k|$ grows, the proportion of endomorphisms with distinct eigenvalues approaches 1. To guarantee that this proportion is at least $1-\alpha$, we need only
that $ \alpha-m\leq |k|^{r-1} \bigl(\alpha|k|-m\bigr). $
  Since $m \geq 1$ and thus the left side is negative, it suffices to take $|k|$ large enough that $\alpha|k|\geq m$.
\end{rem}

Note that by nature of probabilistic algorithms, if \Cref{alg:graded} or \ref{alg:local} fail to produce a nontrivial summand, it does not certify that $M$ is indecomposable. Despite this, a corollary of the proof of \Cref{lem:general} is the following criterion for indecomposability.

\begin{cor}\label{cor:generic}
  Let $M$ be a finitely generated $R$-module. Using \Cref{notation:gensEnd0}, if the coefficients of the characteristic polynomial of the generic endomorphism $A \coloneqq \sum_i \alpha_i A_i$ satisfy \eqref{eq:goodchar} when $p\nmid m$ or \eqref{eq:badchar} when $p\mid m$,
 then $M$ is indecomposable over~$\bar k$.
\end{cor}

While the criterion above is definitive, computing the characteristic polynomial of a $\mu(M)\times\mu(M)$ matrix with entries in $k[\alpha_1\,\dots,\alpha_r]$ is costly.
One can replace the characteristic polynomial with the minimal polynomial, which is easier to compute in practice, and use the same criteria as in \cref{cor:generic} on the coefficients of this smaller polynomial.

The following proposition shows a sufficient conditions which is much easier to check:

\begin{prop}\label{prop:indecomposability}
  Let $M$ be a finitely generated $R$-module. Suppose that either:
  \begin{enumerate}

  \item\label{item:indec-graded} $R$ is graded and $[\End_R(M)]_0$ is 1-dimensional and thus spanned by $\id_M$, or
  \item\label{item:indec-local} $R$ is local and the image of $\End_R(M)$ in $\End_k(M/\m M)$ is 1-dimensional;
  \end{enumerate}
  then $M$ is indecomposable.
\end{prop}
\begin{proof}
  To see \eqref{item:indec-graded}, observe that if $M$ decomposes non-trivially as $M_1\oplus M_2$, then the projections onto each factor are nontrivial degree-zero endomorphisms not equal to the identity, which do not have all entries contained in $\m$. The proof for \eqref{item:indec-local} is analogous.
\end{proof}




\begin{rem}\label{rem:complexity}
  We omit a computational complexity analysis here because without any assumptions about the module there are too many parameters that in any given application may become prominent. We note, however, two specific advantages:
  \begin{itemize}[topsep=6pt,itemsep=0pt]
  \item Using degree-limited syzygy algorithms, $[\End_R(M)]_0$ is easier to compute than $\End_R(M)$, particularly in cases where $M$ can be endowed with a multigrading.
  \item Over a large enough field \Cref{alg:graded,alg:local} are likely to terminate in one non-recursive iteration which requires only a single general endomorphism (see \Cref{rem:graded}).
  \end{itemize}
   While in many applications these advantages significantly improve the computational complexity, in general the worst-case complexity of the algorithms presented here is bounded by the complexity of computing the kernel of the map $f^t\otimes_R M$ where $f^t$ is the transpose of a free presentation $f$ of $M$. Compare with \cite{DX22}, which provides a complexity analysis for a decomposition algorithm with the additional assumption that the generators of $R$ as well as generators and relations of $M$ have distinct multidegrees. This assumption is exceedingly rare in applications from commutative algebra and algebraic geometry, but it is relevant, though not guaranteed, in the context of multiparameter persistence in topological data analysis.
\end{rem}

\section{Decomposing coherent sheaves}\label{sec:coherent}

While the preceding section was written in the language of modules, the standard algebro-geometric dictionary implies that \Cref{alg:graded} can be used to find indecomposable decompositions of coherent sheaves on projective varieties.
In this section, we make a few notes in the standard graded case regarding the relation between the eigenvalues discussed in the previous section (in this section, called \emph{module-theoretic eigenvalues}) with the notion of eigenvalue of an endomorphism of a vector bundle.

Let $X\subset \P^n$ be a projective variety, with ample line bundle $\OO_X(1)$.
Throughout, let $E$ be a vector bundle (i.e., a locally free coherent sheaf) on $X$, and consider an endomorphism $f\in\End_{\OO_X}(E)$. The following result is well-known.

\begin{lem}\label{lem:inj-iff-iso}
  An endomorphism $f\colon E\to E$ is injective if and only if it is an isomorphism.
\end{lem}

For completeness, we give the standard proof (see, e.g., \cite[Exercise~4.1]{Friedman98}).

\begin{proof}
  This requires only that $E$ is a coherent sheaf. We claim that $f\otimes k(x)$ is injective for any $x\in X$, which implies that $f\otimes k(x)$ is also surjective for every $x$ and thus $f$ is surjective.

  To see this, note that $f$ satisfies some minimal-degree monic polynomial, since $\End_{\OO_X}(E)$ is finite-dimensional over $k$. This monic polynomial has nonzero constant term since $f$ is injective. Now, $f\otimes k(x)$ satisfies this same polynomial, so must be injective.
\end{proof}

\begin{dfn}
  Let $r\coloneqq\rank E$, then for $f\in \End_{\OO_X}(E)$, taking top exterior powers yields
  \[\textstyle \bigwedge^r f\colon \bigwedge^r E \to \bigwedge^r E. \]
  Since $\End_{\OO_X}(\bigwedge^r E) = H^0(\OO_X)=k$, the map $\bigwedge^r f$ is multiplication by some $\lambda\in k$; we write $\det f$ for this scalar $\lambda$.
\end{dfn}

\begin{lem}
  Let $x \in X$ be any point (not necessarily closed), with residue field $k(x) = \OO_{X,x}/\m_x$. Then $\det(f\otimes k(x)) =\det f$.
\end{lem}
\begin{proof}
  Note that $\rank( E\otimes k(x)) = \rank E = r$. Thus since $(\bigwedge^r f)\otimes k(x) = \bigwedge^r( f\otimes k(x))$ as $k(x)$-linear maps on $E\otimes k(x)$,
  we have $\det(f\otimes k(x)) =\det(f)\otimes k(x)$, but $\det(f)\in k$ and so is unaffected by reducing modulo $\m_x$.
\end{proof}

\begin{lem}
  $\det f \neq 0 $ if and only if $f$ is injective, in which case $f$ is an isomorphism.
\end{lem}
\begin{proof}
  Let $\ker f \neq 0$. Since localization is exact we have $\ker(f\otimes k(X)) = (\ker f)\otimes k(X)$.  Since $\ker f \subset E$ is torsion-free, the localization map $\ker f \to (\ker f)\otimes k(X)$ is injective and thus $(\ker f) \otimes k(X)\neq 0$. Since the $k(X)$-vector space map $f\otimes k(X)$ is not injective, $\det(f\otimes k(X))= 0$. In particular  $\det f = 0$ as well.

  Conversely, if $\ker f = 0$, then $f$ is an isomorphism, hence an isomorphism on fibers, and hence $\det(f\otimes k(x)) \neq 0$ for any $x$. Thus $\det f\neq 0$.
\end{proof}

\begin{dfn}
  $\lambda \in k$ is an \emph{eigenvalue} of $f \in \End_{\OO_X}(E)$ if $\det(f-\lambda \id_E)= 0$. In other words, the eigenvalues of $f$ are the zeroes of the univariate polynomial $\det(f - \lambda\id _E)$ with coefficients in $k$.
\end{dfn}

Note that $f$ is an isomorphism if and only if $\lambda=0$ is not an eigenvalue of $f$, just as for ordinary linear operators on vector spaces.

Now, let $R$ be the homogeneous coordinate ring of $X\subset \P^n$, and let $M$ be a graded $R$-module such that $\wtilde M=E$. For $\phi\in\End_R(M)$ we will refer to the eigenvalues of $\phi$ (see \Cref{def:eigenvalues}) as \emph{module-theoretic} eigenvalues of $\phi$.

\begin{lem}
  Let $f\colon E\to E$ arise from a map $\phi\colon M\to M$, in the sense that $\wtilde M = E$ and $\wtilde\phi = f$. The eigenvalues of $f$ are a subset of the module-theoretic eigenvalues of $\phi$.
\end{lem}

Note that any endomorphism of $E$ corresponds to an endomorphism of $\Gamma_*(E)$; we may want the freedom to work with other module representatives of $E$ though.

\begin{proof}
  All that needs to be shown is that if $\det(f-\lambda \id_E) =0$, then $\det(\phi-\lambda \id_M)=0$. Replacing $f-\lambda \id_E$ by $f$ and likewise $\phi-\lambda\id_M$ by $\phi$,
  we just need to show that $\det f = 0$ implies $\det\phi = 0$.

  Say $\det f = 0$. If $\det\phi = \det(\phi\otimes R/\m)\neq 0$, then $\phi$ induces a surjection $M/\m M\to M/\m M$, and thus $\phi$ induces a surjection $M\to M$ by Nakayama's lemma. The endomorphism $f$ is thus an isomorphism \cite[Corollary~4.4]{Eisenbud95}, and so $\det f\neq 0$.
\end{proof}

\begin{exa}
  By adding irrelevant summands to $M$ (i.e., summands which sheafify to zero), one can always add extraneous module-theoretic eigenvalues, so the containment of eigenvalues may be proper.
  Note that in general the multiplicities will never be equal, even if one takes $M=\Gamma_*(E)$: if $f$ is multiplication by $\lambda$, then $f$ will have $\rank E$ many eigenvalues with multiplicity, while $\phi$ will have $\mu(M)$ module-theoretic eigenvalues with multiplicity.
\end{exa}

Finally, we may combine the results of this section with \Cref{prop:split-surj} to prove the following proposition.

\begin{prop}
  Suppose $E$ is a vector bundle on a $k$-variety $X$, and $f\in\End_{\OO_X}(E)$ has distinct eigenvalues $\lambda_1,\dots,\lambda_n$. Set $g_i = (f - \lambda_i\id_E)^{\mu_i}$, where $\mu_i$ are module-theoretic multiplicities of the corresponding eigenvalues of $\Gamma_*(f)$, and $g = g_1\circ\dots\circ g_n$, then $E$ has a direct sum decomposition
  \[ E = \ker g_1 \oplus \dots \oplus \ker g_n \oplus \im g. \]
\end{prop}
\begin{proof}
  Take $M = \Gamma_*(E)$, $\psi_i = \Gamma_*(g_i)$, and $\psi = \Gamma_*(g)$; then since each $\lambda_i$ is also a module-theoretic eigenvalue of $M$, \Cref{prop:split-surj} produces a direct sum decomposition
  \[ M = \ker\psi_1 \oplus \cdots \oplus \ker\psi_n \oplus \im\psi; \]
note that sheafification is exact, so the sheaf associated to $\ker \psi_i$ is $\ker \wtilde \psi_i = \ker g_i$ (and this is nonzero), while the sheaf associated to $\im \psi $ is $\im g$.
Thus, taking the sheafification of the decomposition of $M$,
we have the desired direct sum decomposition of $E$.
\end{proof}

\begin{rem}
  We note that it is known already by \cite[Proposition~15]{Atiyah57} that an endomorphism of an indecomposable vector bundle $E$ cannot have two distinct eigenvalues. The utility of the preceding lemma is in producing an explicit direct sum decomposition of $E$.
\end{rem}

\section{Examples}\label{sec:examples}

In this section, we give examples of the kind of calculations and observations \Cref{alg:graded,alg:local} allow us to make.

\begin{exa}[Frobenius pushforward on the projective space $\P^n$]
  Let $S = k[x_0,\dots,x_n]$ be a polynomial ring with $\operatorname{char} k = p$ and $\deg x_i = 1$ and consider the Frobenius endomorphism
  \[ F\colon S\to S \quad \text{given by} \quad f \to f^p. \]
  Hartshorne \cite{Hartshorne1970} proved that for any line bundle $L\in\Pic\P^n$, the Frobenius pushforward $F_*L$ splits as a sum of line bundles. While the following calculations are straightforward to do by hand, they are immediately calculated via our algorithm:
  \begin{align*}
    \text{When }p=3, n=2: \\
    F_*\OO_{\P^2} &= \OO \oplus \OO(-1)^7 \oplus \OO(-2). \\
    \text{When } p=2, n=5: \\
    F_*\OO_{\P^5} &= \OO \oplus \OO(-1)^{15} \oplus \OO(-2)^{15} \oplus \OO(-3), \\
    F_*^2\OO_{\P^5} &= \OO \oplus \OO(-1)^{120} \oplus \OO(-2)^{546} \oplus \OO(-3)^{336} \oplus \OO(-4)^{21}.
  \end{align*}
\end{exa}

\begin{exa}[Frobenius pushforward on toric varieties]
  Let $X$ be a smooth toric variety and consider its Cox ring
  \[ S = \bigoplus_{[D]\in\Pic{X}} \; \Gamma(X, \OO(D)). \]
  Similar to the case of the projective space, B{\o}gvad and Thomsen \cite{Bogvad98,Thomsen00} showed that $F_*L$ totally splits as a direct sum of line bundles for any line bundle $L\in\Pic X$.

  As an example, consider the third Hirzebruch surface $X=\P(\OO_{\P^1}\oplus \OO_{\P^1}(3))$ over a field of characteristic 3. We have, for example, that
  \begin{align*}
    F_*\OO_X      &= \OO_X   \oplus \OO_X(-1,0)^2 \oplus \OO_X(0,-1)^2 \oplus \OO_X(1,-1)^3 \oplus \OO_X(2,-1), \\
    F_*\OO_X(1,1) &= \OO_X^3 \oplus \OO_X(-1,0)   \oplus \OO_X(1,-1)   \oplus \OO_X(1, 0)^2 \oplus \OO_X(2,-1)^2.
  \end{align*}
  In fact, Achinger \cite{Achinger15} showed that the total splitting of $F_*L$ for every line bundle $L$ characterizes smooth projective toric varieties.
\end{exa}

\begin{exa}[Frobenius pushforward on elliptic curves]\label{ex:elliptic}
  Consider the elliptic curve
  \[ X = \Proj \FF_7[x,y,z]/(x^3+y^3+z^3). \]
  This is an ordinary elliptic curve,
 hence $F$-split; thus $\OO_X$ is a summand of $F_* \OO_X$. Over the algebraic closure of $\FF_7$, $F_*\OO_X$ will decompose as $\bigoplus_{p=1}^7 \OO_X(p_i)$, where $p_1,\dots,p_7$ are the 7-torsion points of $X$. (For more on indecomposable summands of $F_*\OO_X$ on ordinary abelian varieties, see \cite{Tango,ST}.)

  However, over $\FF_7$, our algorithm calculates that $F_*\OO_X$ decomposes only as
  $$ F_* \OO_X =\OO_X \oplus M_1\oplus M_2\oplus M_3, $$
  with $M_i$ indecomposable (over $\FF_7$) of rank 2.

  After extending the ground field to $\FF_{49}$, our algorithm calculates the full decomposition
  $$ F_* \OO_X=\bigoplus_{p=1}^7 \OO_X(p_i). $$
  This reflects the fact that the 7-torsion points $p_i$ of $X$, and thus the sheaves $\OO_X(p_i)$, are not defined over $\FF_7$, but rather are defined over $\FF_{49}$.
\end{exa}

\begin{exa}[Frobenius pushforward on Grassmannians]
  Consider the Grassmannian $X = \Gr(2,4)$. We may work over the Cox ring $S$,
  which in this case coincides with the coordinate ring
  \[ S = \frac{k[p_{0,1},p_{0,2},p_{0,3},p_{1,2},p_{1,3},p_{2,3}]}{p_{1,2}p_{0,3}-p_{0,2}p_{1,3}+p_{0,1}p_{2,3}}. \]
  Then in characteristic $p=3$ we have:
  \[ F_*\OO_X = \OO \oplus \OO(-1)^{44} \oplus \OO(-2)^{20} \oplus A^4 \oplus B^4, \]
  where $A$ and $B$ are rank-2 indecomposable bundles (c.f.~\cite{RSVdB22}).
\end{exa}

\begin{exa}[Frobenius pushforward on Mori Dream Spaces]
  Continuing with the theme of computations over the Cox ring, the natural geometric setting is to consider the class of projective varieties known as Mori dream spaces \cite{HK00}.

  For instance, consider $X = \Bl_4\P^2$, the blowup of $\P^2$ at 4 general points. We work over the $\ZZ^5$-graded Cox ring
  \[ S = k[x_1,\dots,x_{10}]/\text{(five quadric Pl\"ucker relations)} \]
  with degrees
  \[
  \left(\!\begin{array}{rrrrrrrrrr}
  0&0&0&0&1&1&1&1&1&1 \\
  1&0&0&0&-1&-1&-1&0&0&0 \\
  0&1&0&0&-1&0&0&-1&-1&0 \\
  0&0&1&0&0&-1&0&-1&0&-1 \\
  0&0&0&1&0&0&-1&0&-1&-1
  \end{array}\!\right).
  \]
  Then in characteristic 2 we have:
  \begin{align*}
    F_*^2\OO_X = {\OO_{X}^{1}}
    &\oplus {\OO_{X}^{2}\ \left(-2,\,1,\,1,\,1,\,1\right)} \oplus {\OO_{X}^{2}\ \left(-1,\,0,\,0,\,0,\,1\right)} \\
    &\oplus {\OO_{X}^{2}\ \left(-1,\,0,\,0,\,1,\,0\right)} \oplus {\OO_{X}^{2}\ \left(-1,\,0,\,1,\,0,\,0\right)} \\
    &\oplus {\OO_{X}^{2}\ \left(-1,\,1,\,0,\,0,\,0\right)} \oplus B \oplus G,
  \end{align*}
  where $B, G$ are rank-3 and rank-2 indecomposable modules, as calculated in \cite{Hara15}.
\end{exa}

\begin{exa}[Frobenius pushforward on cubic surfaces]
  Let $X$ be a smooth cubic surface. Aside from a single exception in characteristic 0, $X$ will be globally $F$-split, so that any $F^e_*\OO_X $ admits $\OO_X$ as a direct summand.
  The other summands of Frobenius pushforwards of $\OO_X$ have yet to be studied, and in particular it is not known whether such rings should have the finite $F$-representation type property.

  The use of our algorithm to compute examples in small $p$ and $e$ suggest the behavior
  $$ F_* \OO_X = \OO_X\oplus M, $$
  with $M$ indecomposable, and furthermore $F_*^e M$ remains indecomposable for all $e\geq 0$. In other words, the indecomposable decomposition of $F^e_* \OO_X$ is
  $$ F_*^e \OO_X \cong \OO_X\oplus M\oplus F_* M\oplus\dots\oplus F_*^{e-1}M. $$
  In particular, this suggests $\OO_X$ will fail to have the finite $F$-representation type property.
  In fact, we believe a similar description holds true for quartic del Pezzos.
\end{exa}


\begin{exa}[Local singularities]\label{exa:local-singularities}
  Let $R = \FF_2[x,y,z]_{(x,y,z)}/(x^2y+xy^2+xyz+z^2)$. In the notation of \cite{Artin77}, this is the $D_4^1$ singularity, which is an exceptional version of the usual rational double point/Du Val singularities appearing in characteristic 2.  In particular, note that $R$ is not homogeneous. If $F_* R$ denotes the Frobenius pushforward of $R$, then the algorithm of \Cref{sec:local-alg} computes the following indecomposable summands:
  $$
  F_*R = R\oplus
  \coker
  \begin{pmatrix}
    x+y+z&z\\
    z&xy
  \end{pmatrix}\oplus\coker \begin{pmatrix}
    y&z\\
    z&x^{2}+xy+xz
  \end{pmatrix}\oplus \coker\begin{pmatrix}
  x&z\\
  z&xy+y^{2}+yz
  \end{pmatrix};
  $$
  there is one free summand and three reflexive modules of rank 1.
\end{exa}

\begin{exa}[Syzygies over Artinian rings]
  In recent work suggested by examples calculated using our algorithm, \cite{CDE24} studied the indecomposable summands of syzygy modules over a Golod ring $(R,\m,k)$ and found previously unexpected recurring behavior. Specifically, the syzygy modules of the residue field are direct sums of only three indecomposable modules: the residue field $k$, the maximal ideal $\m$, and an additional module $N = \Hom_R(\m, R)$.

  Here, we give a concrete example of one such ring. Let $k$ be any field and let $R = k[x,y]/(x^3,x^2y^3,y^5)$ and consider the (infinite) minimal free resolution of the residue field, which has rank $2^n$ in homological index $n$.   The fourth syzygy module of the residue field decomposes (ignoring the grading) as the direct sum
  $$ k^3 \oplus \m^2 \oplus N^3, $$
  and the fifth syzygy module as
  $$ k^8\oplus \m^9 \oplus N^2, $$
  where the module $N$ can be explicitly presented as
  \[ N = \coker
  \begin{pmatrix}
    x^{2}&0&0&0&y^{4}&xy^{3}&0&0\\
    -y&x&y^{3}&0&0&0&0&0\\
    0&0&0&y&-x&0&0&0\\
    0&0&0&0&0&-y&x&y^{2}
  \end{pmatrix} \]

  The use of our algorithm was essential to the observation that beyond the ``guaranteed'' summands of $k$ and $\m$ (which were known to appear by work of \cite{DE23}) only the one additional indecomposable module $N$ appears in the summands of syzygies of $k$.
\end{exa}


\begin{exa}[Symbolic diagonalization]\label{ex:diagonalization}
  An interesting application of our algorithm, suggested by Bernd Sturmfels,  is automated diagonalization of symbolically parameterized matrices. As a simple demonstration, let $R = K[a,b,c,d]$ and consider the following matrix
  \[ A = \begin{pmatrix}
    a&b&c&d\\
    d&a&b&c\\
    c&d&a&b\\
    b&c&d&a
  \end{pmatrix}. \]
  Then the splittings of $\coker A$ over $K = \QQ$ and $K = \QQ(i)$ have the following presentations, respectively:
  \[
  \thinmuskip1mu
  \medmuskip0mu \begin{pmatrix}
    a+b+c+d&0&0&0\\
    0&a-b+c-d&0&0\\
    0&0&a-c&b-d\\
    0&0&b-d&c-a
  \end{pmatrix},\,
  \begin{pmatrix}
    a+b+c+d&0&0&0\\
    0&a-b+c-d&0&0\\
    0&0&a+bi-c-di&0\\
    0&0&0&a-bi-c+di
  \end{pmatrix}.
  \]
  In particular, the change of basis matrices can be extracted from our algorithm.
\end{exa}



\bibliographystyle{alpha}
\bibliography{references.bib}

\end{document}